\documentclass[12pt]{article}

\usepackage{amssymb}
\usepackage{pdfpages}
\usepackage{graphicx}
\usepackage{floatrow}
\usepackage{amsmath}
\usepackage{amsfonts}
\usepackage[colorlinks,linkcolor=blue,citecolor=red]{hyperref}
\hypersetup{
  colorlinks=true,
  urlcolor=blue}
\newtheorem{theorem}{Theorem}%[section]
\newtheorem{corollary}[theorem]{Corollary}
\newtheorem{definition}{Definition}

\newtheorem{example}{Example}[section]

\newtheorem{lemma}[theorem]{Lemma}

\newtheorem{claim}{Claim}
\newtheorem{proposition}[theorem]{Proposition}
\newenvironment{proof}{\noindent {\bf Proof.}}{\hfill\rule{3mm}{3mm}\par\medskip}

\newtheorem{prelem}{{\bf Theorem}}

\newenvironment{oldtheorem}{\begin{prelem}{\hspace{-0.5
em}{\bf}}}{\end{prelem}}

\newtheorem{prelemc}{{\bf Conjecture}}

\title{Hamiltonian cycles in planar cubic graphs with facial $2-$factors, and a new partial solution of Barnette's Conjecture}
\author{\sc Behrooz Bagheri Gh.${}^{a,b}$, Tomas Feder${}^{c}$,  Herbert Fleischner${}^{a}$, \\ Carlos Subi}
\date{}

\begin{document}
\maketitle
  \vspace{-1cm}
  \begin{center}
  	$a$
  	{\small \it Algorithms and Complexity Group}\\
  	{\small \it Vienna University of Technology}\\
  	{\small  \it Favoritenstrasse 9-11,}
  	\vspace*{5mm}
  	{\small \it 1040 Vienna, Austria }\footnote{
  		Research supported in part by FWF Project P27615-N25.} \\
  	$b$
{\small \it Department of Mathematics} \\
{\small \it West Virginia University} \\
{\small \it Morgantown,}
\vspace*{5mm}
{\small \it WV 26506-6310, USA} \\
  	$c$
  	{\small \it Computer Science Department} \\
  	{\small \it Stanford University} \\
  	{\small \it }
  	%\vspace*{5mm}
  	{\small \it Stanford, California 94305, USA} \
  \end{center}
\vspace{-.5cm}
\begin{abstract}
We study the existence of hamiltonian cycles in plane cubic graphs $G$ having a facial $2-$factor $\mathcal{Q}$. Thus hamiltonicity in $G$  is 
transformed into the existence of a (quasi) spanning tree of faces in the 
contraction $G/\mathcal{Q}$. In particular, we study the case where $G$  is the leapfrog extension (called vertex envelope in~(Discrete Math., 309(14):4793--4809, 2009)) of a plane cubic graph $G_0$.
As a consequence we prove hamiltonicity in the leapfrog extension of planar cubic cyclically $4-$edge-connected bipartite graphs. This and other 
results of this paper establish partial solutions of Barnette's Conjecture according to which every $3-$connected cubic planar bipartite graph is hamiltonian. These results go considerably beyond Goodey's result on this topic (Israel J.
Math., 22:52--56, 1975).\\
\\
%
%\vspace{1cm}
{\bf Keywords:}
 Barnette's Conjecture; eulerian plane graph; hamiltonian cycle; spanning tree of faces;  $A-$trail.
\end{abstract}

\section{Introduction and Preliminary Discussion}

Hamiltonian graph theory has its roots in the icosian game which was  introduced 
by W.R. Hamilton in $1857$. However, Kirkman presented his paper 
{\it On the presentation of polyhedra},~\cite{Kirkman}, to the Royal Society 
already in $1855$; and it was published in $1856$.

The early development of hamiltonian graph theory focused to a large
extent on planar cubic graphs; and there are good reasons for this course 
of development. For, in $1884$, Tait  conjectured that every cubic 
$3-$connected planar graph is 
hamiltonian,~\cite{Tait}. 
And Tait knew that the validity of his conjecture would yield a simple 
proof of the Four Color Conjecture. On the other hand, the Petersen graph 
is the smallest non-planar $3-$connected cubic graph which is not hamiltonian,~\cite{Petersen}.
Tait's Conjecture
 was disproved by Tutte in $1946$, who constructed a counterexample with $46$ 
vertices,~\cite{Tutte1946}; other researchers later found even smaller 
counterexamples. However, none of these known counterexamples are 
bipartite. Tutte himself conjectured that every cubic $3-$connected 
bipartite graph is hamiltonian,~\cite{Tutte1971}, but this was shown to be 
false by the construction of a counterexample, 
the Horton graph,~\cite{Horton}.  Barnette  proposed a 
combination of Tait's and Tutte's Conjectures that every counterexample 
to Tait's Conjecture is non-bipartite.\\

\medskip \noindent
{\bf Barnette's Conjecture}~\cite{Barenette}
{\it Every $3-$connected cubic planar bipartite graph is hamiltonian.}\\

This conjecture was verified for graphs with up to
$64$ vertices by Holton, Manvel and McKay,~\cite{Holton}. The conjecture 
also holds for the infinite family of graphs where all faces are either quadrilaterals or  hexagons, as shown by
Goodey,~\cite{Goodey}. Without the assumption of $3-$connectedness, it is NP-complete to decide whether
a $2-$connected cubic planar bipartite graph  is hamiltonian, as shown by Takanori, Takao and Nobuji,~\cite{Takanori}. 

For a more detailed account of the early development of hamiltonian graph theory we refer the interested reader 
to~\cite{Biggs}. 

Given the fact that the existence of hamiltonian cycles is an NP-complete problem (in rather special classes of graphs),
one has to develop ad hoc proof techniques depending on the class of graphs,
whose members are being shown to be hamiltonian.

As for the terminology used in this paper we follow~\cite{Bondy} unless stated explicitly otherwise.
In particular, the subset $E(v)$ of $E(G)$ denotes the set of edges incident to $v\in V(G)$.

\begin{remark}
\begin{description}
\item[1.]
Two edges $e=xy$ and $e^{'}=xy$ are called {\sf parallel edge}s
if the digon $D$ defined by $e$ and $e^{'}$ has no vertices inside.
If two different triangles $T_1,T_2$ have an edge  in common, then 
they have no other edge in common (because of our understanding that 
parallel edges are treated as a single edge), unless there is 
$e_i=xy\in E(T_i),\ i=1,2$, such that $\langle e_1,e_2\rangle$ defines a digon with 
some vertex  inside.
\item[2.] 
Given a $2-$connected plane  graph, we do not distinguish between faces and their face boundaries.
Observe that in planar $3-$connected graphs $H$,
the face boundaries are independent from any actual 
embedding of $H$ in the plane or sphere.
\end{description}

\medskip \noindent
Next, we state some definitions. 

\begin{definition}
A cubic graph $G$ is {\sf cyclically $k-$edge-connected} if at least $k$ edges must be removed to disconnect $G$ either into two components 
 each of which contains a cycle provided $G$ contains two disjoint cycles, or else into two non-trivial components. The {\sf cyclic edge-connectivity} of $G$ is the maximum $k$ such that $G$ is cyclically $k-$edge-connected, denoted by $\kappa{'}_c(G)$.
\end{definition}
\begin{definition}
Let $C$ be a cycle in a plane graph $H$. The cycle $C$ divides 
the plane into two disjoint open domains. The  {\sf interior} $(${\sf exterior}$)$ of $C$ 
is the bounded $($unbounded$)$ domain and is denoted by 
$int(C)\ (ext(C))$. 
By treating parallel edges as a single edge,
we say a cycle $C^{'}$ is {\sf inside} of 
$C$ if $int(C^{'})\subseteq int(C)$.
Moreover,  
a cycle $C$ is said 
{\sf to contain a vertex $v$ inside (outside)}   if $v\in int(C)$ $(v\in ext(C))$. If 
$$int(C)\cap V(H)\neq \emptyset \neq ext(C)\cap V(H),$$ then $C$ is said to be a  {\sf separating} cycle in $H$.
However, in the case of a plane embedding we distinguish between the unbounded or outer face $F_o$ 
$($i.e., $ext(F_o) \cap V(G) = \emptyset)$ and a bounded face $F_i\ ($i.e., $int(F_i) \cap V(G) = \emptyset)$.
\end{definition}

\end{remark}
\begin{definition}
Given a graph $H$ and a vertex $v$, a fixed sequence 
$\langle e_1,\ldots,e_{\deg(v)}\rangle$ of the edges in $E(v)$ 
is called a {\sf positive ordering} of $E(v)$ and is denoted by 
$O^{+}(v)$. If  $H$ is imbedded in some orientable surface, one such $O^{+}(v)$
is given by the counterclockwise cyclic ordering of the edges incident 
to $v$.
\end{definition}

\begin{definition}\label{DEF:A-trail}
Let $H$ be an eulerian graph with a given positive ordering $O^{+}(v)$ 
for each vertex $v\in V(H)$. An eulerian trail $L$ in $H$ is an {\sf $A-$trail} 
if  $\{e_i,e_j\}\subseteq E(v)$ being a pair of consecutive edges  in $L$ implies 
$j=i\pm1 \pmod{deg(v)}$, for every $v\in V(H)$. 
-- As a consequence, in an $A-$trail in a $2-$connected plane graph any two consecutive edges belong to a face boundary.
\end{definition}

\begin{definition}\label{DEF:Radial}
\begin{description}

\item[$(i)$]
Suppose $H$ is a $2-$connected plane  graph. Let  $\mathcal{F}(H)$ be 
the set of faces of $H$. The {\sf radial graph} of $H$ denoted by 
$\mathcal{R}(H)$ is a bipartite graph
with the vertex bipartition $\{V(H),\mathcal{F}(H)\}$ 
such that $xf\in E(\mathcal{R}(H))$ if and only if $x$ is a vertex in the boundary of 
$F\in \mathcal{F}(H)$ corresponding to $f\in V(\mathcal{R}(H))$.

\item[$(ii)$]
Let $U\subseteq V(H)$ and let $\mathcal{T}\subset \mathcal{F}(H)$ be a 
set of bounded faces. 
The {\sf restricted radial graph} 
$\mathcal{R}(U,\mathcal{T})\subset \mathcal{R}(H)$ is defined  by $\mathcal{R}(U,\mathcal{T})={\mathcal{R}(H)}[U\cup \mathcal{T}]$.
\end{description}
\end{definition}

\begin{definition}\label{DEF:Leapfrog}
Let $G$ be a $2-$connected plane graph and let $v$ be a vertex of $G$ with $\deg(v)\ge 3$. Also assume that a sequence 
$\langle e_1,\ldots,e_{\deg(v)}\rangle$, $e_i=u_iv,\ i=1,\ldots,\deg(v)\color{red}{,}$ 
is given by the counterclockwise cyclic ordering of the edges incident to $v$. 
\begin{description}

\item[$(i)$]
  A {\sf truncation} of $v$ is the process of replacing $v$ with a 
cycle $C_v=v_1\ldots v_{\deg(v)}v_1$ and replacing $e_i=u_iv$ with 
$e_i^{'}=u_iv_i$, for $i=1,\ldots,\deg(v)$, in such a way
that the result is a plane graph again. A plane graph obtained from 
$G$ by truncating  all vertices of $G$ is called {\sf truncation of $G$} 
and denoted by $Tr(G)$
subject to the condition that $C_v\cap C_w=\emptyset$ for every pair $\{v,w\}\subset V(G)$.
\item[$(ii)$]
The {\sf leapfrog extension} of the plane graph $G$ is $Tr(G^{*})$ where
$G^{*}$ is the dual of $G$; we denote it by $Lf(G)$.
Alternatively and more formally, the leapfrog extension $Lf(G)$ of a plane 
graph $G$ is $(G\cup \mathcal{R}(G))^{*}$. In the case of cubic $G$,
it can be viewed as obtained from $G$  by replacing every $v\in V(G)$ by
a hexagon $C_6(v)$, with $C_6(v)$ and $C_6(w)$ sharing an edge if and 
only if $vw\in E(G)$; and these hexagons are faces of $Lf(G)$. 
\end{description}
\end{definition}

Next we quote some known results.

\begin{oldtheorem}~$(${\rm\cite[Lemma $2$ and Theorem $3$]{FleischnerEnvelope}}$)$\label{TH:Lf-3-conn.} 
Let $G$ be a plane graph.  The following is  true.

\begin{description}

\item[$(i)$]
If $G$ is connected  and $|E(G)|\ge 2$, then $Lf(G)$ is $2-$connected.

\item[$(ii)$]
If $G$  is  a simple $2-$connected plane graph, then $Lf(G)$ is $3-$connected.
\end{description}
\end{oldtheorem}
\begin{oldtheorem}~$(${\rm\cite[Theorem $25$]{FleischnerEnvelope}}$)$\label{TH:vertexenvelope} 
A plane cubic graph $G$ is the leapfrog extension of a cubic plane graph 
$G_0$ if and only if $G$ has a facial $2-$factor $\mathcal{Q}$, and all 
other face boundaries of $G$ are hexagons.
\end{oldtheorem}

\begin{oldtheorem}~$(${\rm\cite{Payan}}$)$\label{TH:Payan}
Let $G$ be a cyclically $4-$edge-connected  cubic graph of order $n \equiv 2\pmod{4}$.
Then $G$ has an independent set $S$ of order $(n+2)/4$ such that $G[V(G)\setminus S]$ is a tree.
\end{oldtheorem}

\begin{oldtheorem}~$(${\rm\cite[Corollary 15]{FleischnerEnvelope}}$)$\label{TH:LF-ham.} 
If $G$ is a cyclically $4-$edge-connected planar cubic graph of order $n \equiv 2\pmod{4}$, then $Lf(G)$ is hamiltonian.
\end{oldtheorem}

Below we  present a proof of Theorem~\ref{TH:LF-ham.} which relies exclusively on Theorem~\ref{TH:Payan} and differs therefore from the proof in~\cite{FleischnerEnvelope}.

\begin{proof}
 We draw $Lf(G)$ in the plane and draw $G$ so to speak inside of $Lf(G)$ in such a way that $v \in V(G)$ lies inside the corresponding hexagonal face $C_6(v)\subset Lf(G)$ and $vw \in E(G)$ crosses the edge lying in $C_6(v) \cap C_6(w)$ (see Definition~\ref{DEF:Leapfrog}(ii)). 
 
 Now, since $G$ is a cyclically $4-$edge-connected  cubic graph of order $n \equiv 2\pmod{4}$, by Theorem~\ref{TH:Payan} there exists an independent set $S \subset V(G)$ such that  $T =G[V(G)\setminus S]$ is a tree. Now, if we delete
 in $Lf(G)$ those edges of $C_6(x)$ which do not belong to any other $C_6(y)$,  for every $x \in S$, then we obtain the plane graph $G(T)$ covered by the hexagonal faces $C_6(q)$ where $q\in V(G)\setminus S = V(T)$. Letting $K$ be the set of these hexagonal faces, it follows that $T = I(K)$ where $I(K)$ is the intersection graph of $K$. Since for every pair $C_6(t), C_6(u) \in K$ we have $C_6(t) \cap C_6(u) = \emptyset$ or a single edge of $Lf(G)$, and because $I(K)$ is a tree and $V(G(T)) = V(Lf(G))$ by construction, it follows that $G(T)$ has a (unique) hamiltonian cycle which is also a hamiltonian cycle of $Lf(G)$.
\end{proof}

We note in passing that others speak of vertex envelope, or leap frog construction, or leap 
frog operation, or  leap frog transformation (see e.g.~\cite{FleischnerEnvelope,Fowler,Kardos,Yoshida}). 

\begin{definition}
Let $H$ be a  $2-$connected plane graph, let $U\subseteq V(H)$  
and let $\mathcal{T}\subset \mathcal{F}(H)$ be a set of bounded faces
whose boundaries are pairwise edge-disjoint
and such that every vertex of $H$ is contained in some element of $\mathcal{T}$.
 We define a subgraph $H_{\mathcal{T}}$ of $H$  by 
$H_{\mathcal{T}}=H[ \cup_{F\in \mathcal{T}}E(F)]$.
If  $\big|\big\{F\in \mathcal{T}\ :\ x\in V(F)\big\}\big|=\frac{1}{2} \deg_H(x)$
for every $x\in V(H)\setminus U$,
and if $\mathcal{R}(U,\mathcal{T})$ is a tree, then we call $H_{\mathcal{T}}$
a {\sf quasi spanning tree of faces} of $H$, and the vertices in 
$U\ (V(H)\setminus U)$ are called {\sf proper} $(${\sf quasi}$)$ vertices.
If $U=V(H)$, then $H_{\mathcal{T}}$ is called a {\sf spanning tree of faces}.
-- In other words, a spanning tree of faces is a spanning bridgeless cactus whose cycles are face boundaries.
\end{definition}

We observe that if $H$ is a  plane eulerian graph with $\delta(H)\ge 4$
having an $A-$trail $T_{\varepsilon}$, then  $T_{\varepsilon}$ defines 
uniquely a quasi spanning tree of faces as follows 
(see~\cite[pp. $VI.71-VI.77$]{Fleischner}). Starting with a 
$2-$face-coloring of $H$ with colors $1$ and $2$, suppose the outer face 
of $H$ is colored $1$. Then $T_{\varepsilon}$ defines in every 
$v\in V(H)$ a $1-$splitting or a $2-$splitting thus defining a vertex 
partition $V(H)=V_1 \dot{\cup} V_2$ ($ T_{\varepsilon} $ defines 
a $k-$splitting in every $v\in V_k$).
Now, the set $\mathcal{T}$ of all faces colored $2$ defines a quasi spanning tree of
faces $H_{\mathcal{T}}$ with $V_1$ being the set of all quasi vertices of  
$H_{\mathcal{T}}$.
Conversely, a (quasi) spanning tree of faces $H_{\mathcal{T}}$ defines
uniquely an $A-$trail in the subgraph $H_\mathcal{T}$
which is an $A-$trail of $H$ since $\mathcal{T}$ is the set of faces colored $2$.

The aforementioned relation between the concepts of $A-$trail  and 
(quasi) spanning tree of faces is not a coincidence. In fact, it had been shown
(\cite[pp. $VI.112-VI.113$]{Fleischner})
that 

$\bullet$ {\it Barnette's Conjecture is true if and only if every simple $3-$connected eulerian 
triangulation of the plane admits an $A-$trail.}\\ 

We point out, however, that the concept of (quasi) spanning tree of faces
is a somewhat more general tool to deal with hamiltonian cycles in
plane graphs, than the concept of $A-$trails. Below we shall prove 
the existence of (quasi) spanning trees of faces in plane graphs $H$ derived from  plane cubic graphs 
having a facial $2-$factor
(rather than being bipartite - which implies the existense of three facial $2-$factors),
provided the cubic graphs satisfy some extra conditions.
In this context we also want to point out that every simple $4-$connected 
eulerian triangulation of the plane has a   quasi spanning tree of faces (see Corollary~\ref{cor:4con-eulerian-triangulation} below), whereas it is an unsolved 
problem (see~\cite[Conjecture $VI.86$]{Fleischner}), that every simple $4-$connected 
eulerian triangulation of the plane admits an $A-$trail.

Finally observe that we did not include figures in proofs. Instead we elaborated arguments 
to such an extent that the reader himself/herself may draw such figures easily (and in a unique way) as he/she sees fit.
We also wish to point out that this paper is the result of extracting
those results and their proofs of~\cite{Feder} which appear correct to 
all four of us; they have not been published yet
in any refereed journal.
 On top of it, the first author of this paper succeeded in 
developing additional results and their proofs, basing his contribution 
on some of the work in~\cite{Feder}.
Moreover, we relate some of the results of this paper to the theory of 
$A-$trails, as developed in~\cite{Fleischner}.

\section{Hamiltonian cycle from quasi spanning tree of faces}

In what follows \\

\medskip
 \noindent
{\it $G$ always denotes a $3-$connected cubic planar graph 
having a facial $2-$factor $\mathcal{Q}$ $($i.e., a $2-$factor whose cycles are face 
boundaries of $G)$, together with a fixed imbedding in the Euclidean plane; we denote the set of face boundaries of $G$ not in 
$\mathcal{Q}$ by $\mathcal{Q}^{c}$.
 In general, when we say that  a face $F$ is an $\mathcal{X}-$face, we mean that 
  $F\in \mathcal{X}$.
 Let $H$ always denote the reduced graph obtained from $G$ by contracting  the $\mathcal{Q}-$faces 
  to single vertices; i.e., $H=G/\mathcal{Q}$.
  \hfill $(\bf H)$}\\
  
   Suppose $H$ has a quasi spanning tree of 
 faces $H_{\mathcal{T}}$ with proper vertex set $U$.   
 Then the subgraph $H_{\mathcal{T}}$ has a unique $A-$trail which can be transformed    into a hamiltonian cycle $C_G$ of $G$ such that
 the $\mathcal{Q}-$faces 
 corresponding to the  vertices  in $U$ are in $\mathcal{Q}\cap int(C_G)$,
 whereas the faces  of $\mathcal{Q}$ in $\mathcal{Q}\cap ext(C_G)$  
 correspond to  the quasi vertices.
 Moreover, the face of $G$ corresponding to the outer face of $H$ lies in $ext(C_G)$. 

Conversely, suppose $C_G$ is a hamiltonin cycle of $G$ with  outer 
$\mathcal{Q}^{c}-$face in $ext(C_G)$ such that no two $\mathcal{Q}^{c}-$faces sharing an edge
lie in $int(C_G)$. Let $U\subset V(H)$ be the vertex set corresponding to $\mathcal{Q}-$faces in $int(C_G)$. Also, let $\mathcal{T}$ be the set of faces
of $H$ corresponding to $\mathcal{Q}^{c}-$faces in $int(C_G)$.
Since every pair of $\mathcal{Q}^{c}-$faces in $int(C_G)$ has no edge in common by hypothesis, 
$C_G$ can be transformed into an $A-$trail of $H_{\mathcal{T}}$.
Now it is easy to see that $H_{\mathcal{T}}$ is a quasi spanning tree of faces of $H$ whose quasi vertices correspond to the $\mathcal{Q}-$faces in $ext(C_G)$. \\
 
 We summarize the preceding considerations in the following result.
 
\begin{proposition}~$(${\rm\cite[Proposition 1]{Feder}}$)$
\label{PR:1}
Let $G,\mathcal{Q},$ and $H=G/\mathcal{Q}$  be as stated in $(\bf H)$.
  The reduced
graph $H$ has a quasi spanning tree of faces, $H_{\mathcal{T}}$ with face set $\mathcal{T}$, and with
 the external face of $H$ not 
in $\mathcal{T}$ if and only if $G$
has a hamiltonian cycle $C$ with  the external $\mathcal{Q}^{c}-$face lying in $ext(C)$, 
with all $\mathcal{Q}-$faces corresponding to proper vertices of $H_{\mathcal{T}}$ lying in $int(C)$, with 
all $\mathcal{Q}-$faces corresponding to quasi vertices of $H_{\mathcal{T}}$
lying in $ext(C)$, and such that no two $\mathcal{Q}^{c}-$faces sharing an edge are both 
inside of $C$.
\end{proposition}

\begin{example}
In Figure~$\ref{FIG:HamiltonianCycle}$, a $3-$connected cubic planar graph  $G_0$ is given  with a facial $2-$factor $\mathcal{Q}_0=\{v_0v_1v_6v_7v_0,v_2v_3v_{24}v_{25}v_2,v_4v_5v_{13}v_{14}v_4,v_8v_9v_{18}v_{19}v_8,v_{10}v_{11}v_{12}v_{10},\\v_{15}v_{16}v_{17}v_{22}v_{23}v_{15},v_{20}v_{21}v_{26}v_{27}v_{20}\}$.
The hamiltonian cycle $C_0=v_0v_1\ldots v_{27}v_0$  $($bold face lines in $G_0$ in Figure~$\ref{FIG:HamiltonianCycle})$ satisfiies all conditions in Proposition~$\ref{PR:1}$ except the last one; there are two $\mathcal{Q}_0^{c}-$faces inside of $C_0$ sharing the edge
$v_{11}v_{16}$. 
As one sees in the reduced graph $H_0$, the set of faces corresponding  to the $\mathcal{Q}_0^{c}-$faces inside of $C_0$ do not 
correspond to a quasi spanning tree of faces of $H_0$.

\vspace{1cm}
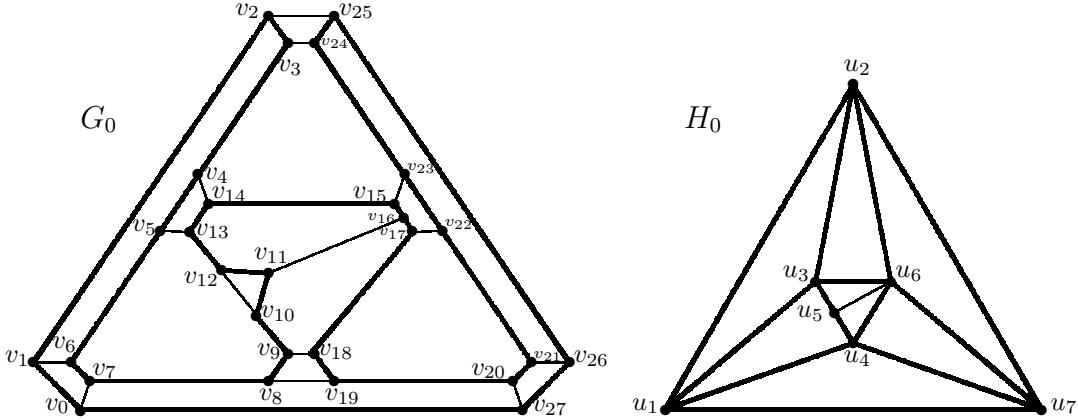
\begin{figure}[ht]

\setlength{\unitlength}{0.125cm}
\vspace{1cm}
\begin{center}

\begin{floatrow} 

\begin{picture}(40,30)
\put(0,30){$G_0$}
\put(0,-0.2){\circle*{1.2}}
\put(-3.75,0){\footnotesize$v_0$}
\put(-5,5){\circle*{1.2}}
\put(-8,5){\footnotesize$v_1$}
\put(20,41.9){\circle*{1.2}}
\put(16.5,42){\footnotesize$v_2$}
\put(22.1,39){\circle*{1.2}}
\put(21,36){\footnotesize$v_3$}
\put(12.5,25){\circle*{1.2}}
\put(13.2,24.5){\footnotesize$v_4$}
\put(8.5,19){\circle*{1.2}}
\put(5.5,19){\footnotesize$v_5$}
\put(-1,5){\circle*{1.2}}
\put(-3,6.5){\footnotesize$v_6$}
\put(1,3){\circle*{1.2}}
\put(1.3,3.7){\footnotesize$v_7$}
\put(20,3){\circle*{1.2}}
\put(19,1.){\footnotesize$v_8$}
\put(22.1,5.9){\circle*{1.2}}
\put(18.8,5.8){\footnotesize$v_9$}
\put(18.7,9.9){\circle*{1.2}}
\put(19.2,9.8){\footnotesize$v_{10}$}
\put(20,14.5){\circle*{1.2}}
\put(18.5,16){\footnotesize$v_{11}$}
\put(15,14.8){\circle*{1.2}}
\put(11,13.5){\footnotesize$v_{12}$}
\put(11.6,18.9){\circle*{1.2}}
\put(12.2,18.6){\footnotesize$v_{13}$}
\put(13.6,21.9){\circle*{1.2}}
\put(14,22.5){\footnotesize$v_{14}$}
\put(33.4,21.9){\circle*{1.2}}
\put(29,22.5){\footnotesize$v_{15}$}
\put(34.4,20.4){\circle*{1.2}}
\put(30.5,20.2){\tiny$v_{16}$}
\put(35.3,19){\circle*{1.2}}
\put(31.7,18.5){\tiny$v_{17}$}
\put(24.9,5.9){\circle*{1.2}}
\put(25.3,5.75){\footnotesize$v_{18}$}
\put(27,3){\circle*{1.2}}
\put(26,1.){\footnotesize$v_{19}$}
\put(46,3){\circle*{1.2}}
\put(41.6,3.7){\footnotesize$v_{20}$}
\put(48,5){\circle*{1.2}}
\put(48.2,5.5){\tiny$v_{21}$}
\put(38.5,19){\circle*{1.2}}
\put(38.7,19.5){\tiny$v_{22}$}
\put(34.5,25){\circle*{1.2}}
\put(34.7,25.5){\tiny$v_{23}$}
\put(24.9,39){\circle*{1.2}}
\put(25.5,38.8){\tiny$v_{24}$}
\put(27,41.9){\circle*{1.2}}
\put(27.5,42){\footnotesize$v_{25}$}
\put(52,5){\circle*{1.2}}
\put(52.5,5){\footnotesize$v_{26}$}
\put(47,0){\circle*{1.2}}
\put(48.1,-0.2){\footnotesize$v_{27}$}

\thicklines
\linethickness{0.15mm}
\qbezier(27,41.9)(20,41.9)(20,41.9)
\qbezier(1,3)(1,3)(0,0)
\qbezier(-1,5)(-1,5)(-5,5)
\qbezier(46,3)(46,3)(47,0)
\qbezier(48,5)(48,5)(52,5)
\qbezier(20,3)(20,3)(27,3)
\qbezier(24.9,5.9)(24.9,5.9)(22.1,5.9)

\qbezier(22.1,39)(22.1,39)(24.9,39)
\qbezier(38.5,19)(38.5,19)(35.4,18.9)
\qbezier(33.4,21.9)(33.4,21.9)(34.5,25)

\qbezier(22.1,5.9)(22.1,5.9)(11.6,18.9)
\qbezier(8.5,19)(8.5,19)(11.6,18.9)
\qbezier(12.5,25)(12.5,25)(13.6,21.9)

\qbezier(20,14.5)(20,14.5)(34.4,20.4)

\thicklines
\linethickness{0.45mm}

\qbezier(0,0)(0,0)(47,0)
\qbezier(0,0)(0,0)(-5,5)
\qbezier(20,41.9)(-5,5)(-5,5)
\qbezier(47,0)(47,0)(52,5)
\qbezier(27,41.9)(52,5)(52,5)
\qbezier(1,3)(1,3)(-1,5)
\qbezier(46,3)(46,3)(48,5)

\qbezier(22.1,39)(22.1,39)(20,41.9)
\qbezier(24.9,39)(24.9,39)(27,41.9)
\qbezier(22.1,39)(22.1,39)(-1,5)
\qbezier(24.9,39)(24.9,39)(48,5)

\qbezier(20,3)(20,3)(1,3)
\qbezier(27,3)(27,3)(46,3)
\qbezier(20,3)(20,3)(22.1,5.9)
\qbezier(27,3)(27,3)(24.9,5.9)

\qbezier(33.4,21.9)(33.4,21.9)(35.4,18.9)
\qbezier(24.5,5.9)(24.5,5.9)(35.4,18.9)
\qbezier(0,0)(0,0)(-5,5)
\qbezier(0,0)(0,0)(-5,5)

\qbezier(33.4,21.9)(33.4,21.9)(13.6,21.9)
\qbezier(13.6,21.9)(13.6,21.9)(11.6,18.9)

\qbezier(22.1,5.9)(22.1,5.9)(18.7,9.9)
\qbezier(11.6,18.9)(11.6,18.9)(15,14.8)
\qbezier(20,14.5)(20,14.5)(18.7,9.9)
\qbezier(20,14.5)(20,14.5)(15,14.8)
\end{picture}

\hspace{2.5cm}

\begin{picture}(40,30)
\put(2,30){$H_0$}
\put(0,0){\circle*{1.2}}
\put(-3.5,0){\footnotesize$u_1$}

\put(20,34.6){\circle*{1.2}}
\put(19,36){\footnotesize$u_2$}
\put(16,13.6){\circle*{1.2}}
\put(12.5,14.1){\footnotesize$u_3$}

\put(20,7.05){\circle*{1.2}}
\put(19.,5.){\footnotesize$u_{4}$}
\put(18,10.325){\circle*{1.2}}
\put(14,9.8){\footnotesize$u_5$}
\put(24,13.6){\circle*{1.2}}
\put(24.5,14.1){\footnotesize$u_6$}

\put(40,0){\circle*{1.2}}
\put(41,0){\footnotesize$u_7$}

\qbezier(24,13.6)(24,13.6)(18,10.325)

\thicklines
\linethickness{0.45mm}

\qbezier(0,0)(0,0)(40,0)
\qbezier(0,0)(0,0)(20,34.6)
\qbezier(0,0)(0,0)(20,7.05)
\qbezier(0,0)(0,0)(16,13.6)

\qbezier(40,0)(40,0)(20,34.6)
\qbezier(40,0)(40,0)(20,7.05)
\qbezier(40,0)(40,0)(24,13.6)

\qbezier(20,34.6)(20,34.6)(16,13.6)
\qbezier(20,34.6)(20,34.6)(24,13.6)

\qbezier(20,7.05)(20,7.05)(16,13.6)
\qbezier(20,7.05)(20,7.05)(24,13.6)
\qbezier(16,13.6)(16,13.6)(24,13.6)

\end{picture}
  \end{floatrow}
\end{center}
\caption{\small\it 
A hamiltonian cycle $C_0=v_0v_1\ldots v_{27}v_0$ in $G_0$ and its corresponding trail
$u_1u_2u_3u_1u_4u_5u_3u_6u_4u_7u_6u_2u_7u_1$ in $H_0$.}
\label{FIG:HamiltonianCycle}
\end{figure}
\end{example}

We return now to our general considerations.
Suppose all $\mathcal{Q}^{c}-$faces of $G$ are either quadrilaterals or 
 hexagons, while the $\mathcal{Q}-$faces are arbitrary.
Suppose the reduced graph $H$ has a triangle $T$ that contains at least 
one vertex in $int(T)$, such that $int(T)$ does not contain a separating 
digon.

We shall successively simplify the inside of the triangle $T$, while 
preserving the property
that there is no separating digon inside of $T$, but allowing the 
presence of separating triangles
inside of $T$, but with the following requirement. 
In what follows we delete loops (but not multiple edges)
 which may arise when contracting a  triangle $T^{'}\subset T$ 
 (such loop may arise
 when $e\in E(T^{'})$ is a multiple edge).
Also, when speaking of a digon or triangle $T^{'}$ not being a face boundary, we mean
that $int(T^{'})$ contains at least one vertex. 

Let $\mathcal{A}$ be the set of all separating triangles in $H$.
Define a relation $\preccurlyeq$ on $\mathcal{A}$ in the following way. 
$T_2 \preccurlyeq  T_1$ if and only if $int(T_2)\subset int(T_1)$, for every $T_1,T_2\in \mathcal{A}$.
This relation is a partial order. 

Suppose $T_1$ and $T_2$ are distinct elements of $\mathcal{A}$ and  
$T_2 \preccurlyeq  T_1$. We say $T_2$ is a {\sf direct successor} of 
$T_1$ if there is no triangle $T_3\in \mathcal{A}$
distinct from $T_1$ and $T_2$ such that $T_2 \preccurlyeq  T_3\preccurlyeq  T_1$.

Note that by planarity, every separating triangle is a direct successor of at most one triangle.

 At all steps in the simplification of the inside of the triangle $T$, we shall require that no triangle $T_1$, $int(T_1)\subseteq int(T)$, 
has three distinct direct successors $T_2$,
$T_2^{'}$, and $T_2^{''}$.
We define the {\sf invariant property} for $T$ to be such that 
$T$ and every triangle inside of $T$ has at most two distinct direct successors and
there is 
no separating digon 
inside of $T$.
In particular, any bounded facial triangle has the invariant property.
Note that if the triangle $T$ has the invariant property, then every triangle  inside of $T$ also has the invariant property (this is sort of a ``relative hereditary property'').
We say that a graph $H$ has the invariant property if every triangle in $H$ (and the outer face of $H$ if it is a triangle) satisfies
the invariant property.

The following theorem is of a more technical nature and is key to the subsequent results.

\begin{theorem}~$(${\rm\cite[Lemma 1]{Feder}}$)$
\label{LE:1}
Let $G,\mathcal{Q},$ and $H=G/ \mathcal{Q}$  be as stated in $(\bf H)$ and let $T\subset H$
be a triangle containing at least two vertices inside. If  $T$ satisfies the invariant property, then it
is possible to select a triangular face $T^{'}$ such that 
$int(T^{'})\subset int(T)$ 
and $|V(T)\cap V(T^{'})|\le 1$,
and after contracting $T^{'}$ to a single vertex
  $T$ will still satisfy the invariant property.
\end{theorem}

\begin{proof}{
Suppose  $T$ satisfies the invariant property.
Let $\mathcal{D}$ be the set of all separating triangles $T^{'}$ inside of $T$ 
such that no triangle inside of $T^{'}$ is separating. That is, $T^{'}$ 
has no direct successors. 
Moreover, $int(T^{'})$ does not contain a separating digon because 
$T$ satisfies the invariant property by supposition.
Observe that $\mathcal{D}\subset \mathcal{A}$, and $T^{'}\in \mathcal{D}$
 corresponds to a sink in the  Hasse diagram of 
 $(\mathcal{A},\preccurlyeq)$. We have two cases.\\
 
\noindent
{\bf Case 1.} There exists a triangle $T_1\in \mathcal{D}$
whose interior contains at least two vertices. 
Set $T_1 = v_1v_2v_3v_1$.\\

In this case,
$v_1$ has at
least two distinct neighbors $v_4$ and $v_5$ inside of $T_1$.
 For if $v_1$ has no such neighbors,
then $v_1$ belongs to a triangle inside of $T_1$ that has an edge 
$v_2v_3$ parallel to the corresponding edge of $T_1$,
contrary to the assumption that there is no separating 
digon inside of $T$ (since $T$ satisfies the invariant property). And if $v_1$
has precisely one such neighbor $v_4$ inside of $T_1$, then for 
$v_4\in N(v_1)\setminus \{v_2,v_3\}$, the triangle
$v_2v_3v_4v_2\subset  int(T_1)$  is separating, contrary to the choice of $T_1\in \mathcal{D}$.

We may then choose $v_4$ and $v_5$ so that $v_2,v_4,v_5$ are consecutive
neighbors of $v_1$, and  contract the triangle $T_2=v_1v_4v_5v_1$.

\begin{claim}\label{clm:1}
By contracting  $T_2$, the  triangle $T$ still satisfies the invariant property.
\end{claim}

Set $H^{'}=H/{T_2}$. We note that in $H^{'}$, the triangle
$T_1$ will not contain any separating digon inside
since such a digon would derive from a  separating triangle inside $T_1$, contrary to 
the choice of $T_1\in \mathcal{D}$.

Next we  show that after the contraction of $T_2$, the triangle $T_1$ has at most two direct successors in $H^{'}$.

In $H^{'}$  however, there may appear separating 
triangles in $int(T_1)$. Such triangles derive from quadrilaterals 
$Q_1=v_1v_4v_6v_7v_1$, $Q_2=v_1v_ 5v_8v_9v_1$, and 
$Q_3=v_4v_5v_{10}v_{11}v_4$ in $H$, which contain some vertices inside other than
$v_4,v_5$. 
 Note that, there is no quadrilateral $Q$ in $H$ containing two edges of $T_2$ and containing a vertex $x \in int(Q)$. Otherwise, such a $Q$ would imply the existence of a separating triangle inside of $T_1$, contrary to the choice of $T_1\in \mathcal{D}$.

 If $v_2=v_6$, then  $v_3=v_7$, since there is no separating triangle inside of $T_1$. In this case, the triangle deriving from $Q_1=v_1v_4v_6v_7v_1$ is
 $T_1$, since we treat parallel edges as a single edge. Thus, suppose that $v_2\neq v_6$.  Moreover, we have  
$v_2\neq v_7\neq v_5\neq v_6$,  $v_8\neq v_4\neq v_9,$ and $v_{11}\neq v_1$, respectively; otherwise we would have 
a separating digon in $int(T_1)\cap H$
which contradicts 
the invariant property of $T$ (since $T_1 \preccurlyeq T$), or a separating triangle in $int(T_1)\cap H$, 
 contradicting the choice of $T_1\in \mathcal{D}$.

The
quadrilaterals of the same type as $Q_2$ are of two kinds:  
first, $int(Q_2)$  contains $v_4$ (in which case $v_9=v_2$), or
else $int(Q_2)$ does  not 
contain $v_4$ in which case it cannot have chords $v_1v_8$ or $v_5v_9$ inside; otherwise, there was 
a separating triangle in $H$ inside of $T_1$, again a 
contradiction to the choice of $T_1\in \mathcal{D}$.
 This implies that for all such quadrilaterals $Q_2^{*}$ and 
 $Q_2^{**}$ containing $v_4$ we have either $Q_2^{*}\subset Q_2^{**}$
or $Q_2^{**}\subset Q_2^{*}$, and the same holds for quadrilaterals not 
containing $v_4$. Hence, let $Q_2^{'}$ be the quadrilateral not containing $v_4$, but  all 
other quadrilaterals of its kind are
contained in $int(Q_2^{'})$; and let $Q_2^{''}$ be the quadrilateral  containing $v_4$, but all other quadrilaterals  of its kind 
are contained in $int(Q_2^{''})$.

The analogous properties hold for the quadrilaterals of the same type as
$Q_1$, but these are 
of only one kind, namely the interior of $Q_1$ containing $v_5$, 
otherwise $v_7 = v_2$,  
contrary to what has been said above.
Let $Q_1^{'}$ be the quadrilateral of the same type as $Q_1$ 
containing $v_5$ and with all  quadrilaterals 
of the same type as $Q_1$ contained in $int(Q_1^{'})$.

The analogous properties also hold for the quadrilaterals of the same 
type as $Q_3$, but 
these are again of
only one kind, namely the interior of $Q_3$ does not contain $v_1$, since they are contained in the triangle $T_1$. Let $Q_3^{'}$ be the quadrilateral not containing $v_1$ but 
with all  quadrilaterals of the same type as $Q_3$ contained in $int(Q_3^{'})$.

That is, $Q_1^{'},Q_2^{'},Q_2^{''}$, and $Q_3^{'}$ are the respective outermost quadrilaterals of their respective kinds.

Let $T_Q\subset H^{'}$ be the triangle deriving from the quadrilateral 
$Q\in \{Q_1^{'},Q_2^{'},Q_2^{''},Q_3^{'}\}$ after contraction of $T_2$ to 
the single vertex $v^{*}$ where $Q_1^{'}=v_1v_4v_6^{'}v_7^{'}v_1$, 
$Q_2^{'}=v_1v_5v_8^{'}v_9^{'}v_1$,
$Q_2^{''}=v_1v_5v_8^{''}v_9^{''}v_1$, and
$Q_3^{'}=v_4v_5v_{10}^{'}v_{11}^{'}v_4$.

Assume that $Q^{'}_1, Q^{'}_2, Q^{''}_2,$  and $Q^{'}_3$ exist. We first
show that $T_{Q_2^{'}} \preccurlyeq  T_{Q_1^{'}}$ and symmetrically,
$T_{Q_3^{'}} \preccurlyeq  T_{Q_1^{'}}$ and $T_{Q_3^{'}} \preccurlyeq  T_{Q_2^{''}}$ (defining the partial order $\preccurlyeq$ in $H^{'}$ as we did in $H$).  
Subsequently we shall conclude that at most one of $Q^{'}_1$ and $Q^{''}_2$ exist.

Suppose  $int(Q_2^{'})\nsubseteq  int(Q_1^{'})$. Since 
$v_5\in int(Q_1^{'})$,  it follows from the supposition that $v_9^{'}\in ext(Q_1^{'})$
(observe that above we concluded $v_6\ne v_5\ne v_7$). Therefore, we have two possibilities for $v_8^{'}$.

\begin{description}
\item[$\bf (1)$] $v_8^{'}=v_6^{'}$.

In this case $v_7^{'}\in int(Q_2^{'})$, since 
 $v_9^{'}\in ext(Q_1^{'})$. Therefore,  the quadrilateral $v_1v_4v_6^{'}v_9^{'}v_1$ contains
properly $Q_1^{'}$, contradicting the definition of $Q_1^{'}$.

\item[$\bf (2)$]  $v_8^{'}=v_7^{'}$.

In this case  there is a separating triangle in $int(T_1)\cap H$ which is either $v_1v_5v_7^{'}v_1$
or $v_1v_7^{'}v_9^{'}v_1$; this contradicts the choice of 
$T_1\in \mathcal{D}$.
\end{description}

Thus, $int(Q_2^{'})\subseteq  int(Q_1^{'})$. Therefore,
$T_{Q_2^{'}} \preccurlyeq  T_{Q_1^{'}}$.

Note that, $int(T_2) \subseteq int(Q_1^{'})$. 
Suppose $int(Q_3^{'}) \subseteq int(Q_1^{'})$ does not hold. Since $Q_1^{'}$ 
is the outermost quadrilateral of its kind the cases $v_6^{'}= v_{11}^{'}$ or $v_7^{'}= v_{10}^{'}$ cannot happen; therefore,  either $v_6^{'} = v_{10}^{'}$ yielding two triangles  $v_4v_{10}^{'}v_{11}^{'}v_4$ and $v_4v_5v_{10}^{'}v_4$ at least one of which is a separating triangle in $H$;
or $v_7^{'} = v_{11}^{'}$, yielding a  triangle  $v_1v_4v_{11}^{'}v_1$  which is separating in $H$.
Each of the above cases yields a separating triangle in $int(T_1)$ contradicting the choice of $T_1$. Thus $int(Q_3^{'})$ does not contain any vertex or edge of $Q_1^{'}$, hence $int(Q_3^{'}) \subseteq int(Q_1^{'})$  and thus 
$T_{Q_3^{'}} \preccurlyeq   T_{Q_1^{'}}$. Likewise we conclude that $T_{Q_3^{'}} \preccurlyeq  T_{Q_2^{''}}$.

Now consider $Q_1^{'}$ and $Q_2^{''}$. In general, we have  
$\{v_1,v_2,v_5,v_6^{'}\}\subseteq N(v_4)$.
Since $v_9^{''}=v_2$, we have two possibilities for $v_8^{''}$ as we had with respect to 
$v_8^{'}$ above. 
\begin{description}
\item[$\bf (1)$] $v_8^{''}=v_6^{'}$.

 In this case there is a separating
triangle 
$v_1v_2v_4v_1$ or $v_2v_4v_6^{'}v_2$ or $v_4v_5v_6^{'}v_4$
in $int(T_1)\cap H$; this
contradicts the choice of $T_1\in \mathcal{D}$.

\item[$\bf(2)$] $v_8^{''}=v_7^{'}$.

In this case  we have a separating 
 triangle $v_1v_2v_8^{''}v_1$ in $int(T_1)\cap H$; this cotradicts the choice of 
$T_1\in \mathcal{D}$.
\end{description}
Thus, at most one of $Q_1^{'}$ and $Q_2^{''}$ exists. Therefore,

$\bullet$ {the above relations  
$T_{Q_2^{'}} \preccurlyeq  T_{Q_1^{'}}$,
$T_{Q_3^{'}} \preccurlyeq  T_{Q_1^{'}}$, $T_{Q_3^{'}} \preccurlyeq  T_{Q_2^{''}}$ and the fact that at most one of $Q_1^{'}$ and $Q_2^{''}$ exists preclude that $T_1$ has three or more direct successors in $H^{'}$;} 

$\bullet$ {if $T_1$ has exactly one direct successor $T_Q$ in $H^{'}$, then 
$T_Q\in \{T_{Q_1^{'}},T_{Q_2^{'}},T_{Q_2^{''}},T_{Q_3^{'}}\}$;}

$\bullet$ {if $T_1$ has two direct successors in 
$H^{'}$, then they are either
$T_{Q_2^{'}}$ and $T_{Q_2^{''}}$, or $T_{Q_2^{'}}$ and $T_{Q_3^{'}}$.}
\\

Note that every triangle deriving from a quadrilateral of the same type 
as $Q_2$ not containing $v_4$, or
every triangle deriving from a quadrilateral of the same type 
as  $Q_3$ has at most one direct successor
deriving from a quadrilateral of its respective type.\\

Every triangle deriving from a quadrilateral of the same type 
as $Q_2$  containing $v_4$  has at most one direct successor
deriving from either a quadrilateral of its type or a quadrilateral of the same  type as $Q_3$.\\

Every triangle deriving from a quadrilateral of the same type 
as $Q_1$  containing $v_5$  has either at most one direct successor
(deriving from either a quadrilateral of its type or from a quadrilateral of the same  type as $Q_2$ not containing $v_4$, or from a quadrilateral of the same  type as $Q_3$) or 
at most two direct successors deriving from  two quadrilaterals,
one
 of the same  types as $Q_2$ not containing $v_4$ and one of the same  type as $Q_3$.\\

Thus, $T$ still satisfies the invariant property in $H^{'}$.
This finishes the proof of Claim~\ref{clm:1} and thus finishes the
consideration of Case 1. \\

\noindent
{\bf Case 2} The interior of every member of $\mathcal{D}$ contains 
precisely one vertex.\\

In this case, there is a triangle $T_1$ (possibly $T_1=T$)
satisfying the invariant property and such
that either it has one direct successor $T_2\in \mathcal{D}$ or
it has  two direct successors 
 $T_2,T_3\in \mathcal{D}$. That is, there exist at most two separating triangles $-T_2$ or $T_2$ and $T_3-$ in $int(T_1)\cap H$. 
Thus, 
\\

{\it if a triangle inside of $T_1$ different from $T_2$ and $T_3$ shares an edge with some
$T_i$, $i\in \{1,2,3\}$, then it is a face boundary. \hfill $(*)$}\\

\noindent
{\bf Subcase 2.1} $E(T_i)\cap E(T_j)=\emptyset$ for $i,j\in \{1,2,3\}$, 
$i\neq j$, and thus $|V(T_i)\cap V(T_j)|\le 1$ since we treat parallel edges as a single edge.\\

By contracting    any triangle inside of $T_2$,
we will not  create a separating digon since
such a digon  would derive from a
separating triangle $T_0$ 
 in $int(T_1)\cap  H$  with  $int(T_0)\cap int(T_2)=\emptyset$ (since $T_2$ is a direct successor of $T_1$) and
$|V(T_0)\cap V(T_2)|=2$; so by the assumption of Subcase 2.1,
$T_0$ is distinct from $T_1$
and $T_3$.
Thus, $T_1$ 
would contain a 
distinct direct successor other than $T_2$ and $T_3$ in $H$, contradicting the choice of $T_1$. \\

\noindent
{\it Note that if a quadrilateral $Q\subset H$ shares two edges 
with $T_2$, then the contraction of  any 
triangle inside of $T_2$  
will not  transform $Q$ into a  separating triangle or a  separating digon
$($otherwise, either $T_1$ would contain in $H$ a distinct separating triangle other than $T_2$ and $T_3$ in $H$ or $T_2$ would not be a direct successor of $T_1$ in $H$- which contradicts the choice of $T_1$ or $T_2)$. \hfill $(**)$}\\

Let $v_0$ be the single vertex inside of $T_2 = v_1v_2v_3v_1$. We must 
again consider quadrilaterals 
$Q_1=v_1v_2v_4v_5v_1$, $Q_2=v_1v_3v_6v_7v_1$, and $Q_3=v_2v_3v_8v_9v_2$
which contain at least one vertex inside other than $v_0$, and
such that $\{v_1,v_2,v_3\}\cap\{v_4,\ldots,v_9\}=\emptyset$.
This equation derives from $(**)$ above and from the invariant property. Moreover, $int(Q_i)\subseteq int(T_1), i=1,2,3$
(otherwise, either $T_1$ would contain a distinct separating triangle  other than $T_2$ and $T_3$ in $H$ - which contradicts the choice of $T_1$ -
or the contraction of any triangle in $int(T_2)$ would not transform that $Q_i$ into a separating digon or separating triangle inside of $T_1$).

\begin{claim}\label{clm:2}
There  cannot exist simultaneously four
 quadrilaterals $Q_i,\ 1\le i\le 3,$ and $Q_1^{'}=v_1v_2v_4^{'}v_5^{'}v_1$
 where $v_{4-i}\in int(Q_i)$ and $v_3\notin int(Q_1^{'})$ but $int(Q_1^{'})\cap V(H)\neq \emptyset$. 
\end{claim}

Supposing that Claim 2 fails and starting with a fixed $Q_1$, we consider all possibilities for $Q_2$ vis-a-vis $Q_1$.

Suppose $v_6=v_5$. Then the triangle $T^{*}=v_1v_5v_7v_1$ is separating.
Since $T_2\subset int(T^{*})$, and $T_2$ is a direct successor of $T_1$ by the choice of $T_1$ and $Q_i \subset T_1$, for $1\le i\le 3$, we have $T^{*}=T_1$.
Thus, $v_1$ is a vertex of $T_1$ and,  consequently,  $v_1\notin int(Q_3)$, which is a contradiction.

Suppose $v_7=v_4$. Then, an analogous reasoning yields
$T_1=v_1v_4v_5v_1$ and the same conclusion as above holds $(v_1\notin int(Q_3))$.

Suppose $v_7=v_5$. Since $T_2\subset v_1v_5v_1$, therefore  $v_1v_5v_1$ is a separating digon
which contradicts the invariant property of $T_1$.

Suppose $v_6=v_4$. If $v_8=v_7$, then $v_3v_6v_8v_3$ would be a separating triangle in $int(T_1)$. Thus, $v_8\neq v_7$. Therefore, $v_9=v_7$. If $v_8=v_4$, then $v_4v_7v_4$ would be a separating digon in $int(T_1)$ (since $v_1\in int(Q_3)$,) which is a contradiction. Thus, $v_8=v_5$. Since $T_2$ is a direct successor of $T_1$ and 
$v_4v_5v_7v_4 \preccurlyeq T_1$ is a separating triangle containing $T_2$, we have $T_1=v_4v_5v_7v_4$, and either

 \begin{description}
\item[i)] the quadrilateral 
$Q_1^{'}$ would be 
 inside the triangle $v_1v_2v_7v_1$ which is a face boundary by $(*)$ 
 (since it shares the edge $v_1v_2$ with $T_2$), and this is 
impossible; or

\item[ii)] $v_5^{'}$ would be inside  the triangle 
$v_1v_5v_7v_1$ which   shares the edge $v_5v_7$ with $T_1$. Thus by $(*)$, the triangle 
$v_1v_5v_7v_1$ is a face boundary and this is a contradiction again; or

\item[iii)]  $v_4^{'}=v_7=v_9$ and 
$v_5^{'}=v_5=v_8$. Note that the quadrilateral 
$Q_1^{'}=v_1v_2v_4^{'}v_5^{'}v_1$ contains some vertex inside;
on the other hand  $v_1v_4^{'}$ devides  $Q_1^{'}$ into two triangles  each of which shares an edge with $T_1$ or $T_2$. Thus by $(*)$, $int(Q_1^{'})\cap V(H)=\emptyset$,  which contradicts the supposed existence of $Q_1^{'}$. 
\end{description}
Claim~\ref{clm:2} now follows.

 Therefore, 
 we may assume
 without loss of generality
that either there is no quadrilateral $Q_1$ containing $v_3$ in its 
interior, or there is
no quadrilateral $Q_1^{'}$ not
containing $v_3$ in its 
interior such that after identifying $v_1$
and $v_2$ a new separating
triangle  arises. 

The contraction of the 
triangle $v_0v_1v_2v_0$  creates only a sequence of triangles with pairwise containment 
involving the new vertex
$v_1\equiv v_2$, apart from the triangle $T_3$, thus preserving the 
property that 
no triangle has three direct successors. Having shown at the begining of 
this subcase that the contraction of $v_0v_1v_2v_0$ does not
create a separating digon, we now conclude that the invariant property is being preserved.  \\

\noindent
{\bf Subcase 2.2} $T_2$  shares one edge with $T_1$
where $T_2=v_1v_2v_3v_1$ as before with $int(T_2)\cap V(H)=\{v_0\}$.
Without loss of generality $E(T_1)\cap E(T_2)=\{v_1v_3\}$.\\

In this subcase $|E(T_2)\cap E(T_3)|\le 1$;
otherwise $E(T_2)\cap E(T_3)=\{v_1v_2,v_2v_3\}$ and then the third edge of $T_3$ would be an edge parallel to $v_1v_3$ in $T_1\setminus int(T_2)$. Since  $T_3$ is a direct successor of $T_1$ and $int(T_3)\cap V(H)\ne \emptyset$ we would have the  separating digon $v_1v_3v_1$ in $int(T_1)$,
a contradiction to the invariant property satisfied by $T_1$. Therefore, without loss of generality suppose that $v_1v_2\notin E(T_3)$.

As before we consider the quadrilateral $Q_1=v_1v_2v_4v_5v_1$ with $(int(Q_1)\setminus \{v_0\})\cap V(H)\neq \emptyset$ and $v_3\notin int(Q_1)$.

If $v_3=v_5$, then
 in $H$, either the digon $v_1v_3v_1$  is a separating digon in $T_1$, which is a contradiction,
 or the triangle $v_2v_3v_4v_2\preccurlyeq T_1$ is separating 
(because there is no separating triangle inside $T_2$), so $T_3=v_2v_3v_4v_2$.
Thus,  the contraction of $v_0v_1v_2v_0$  yields a sequence of triangles with pairwise 
containment involving $v_1 = v_2$ and not containing $v_3$ inside nor in their vertex sets. Therefore, the triangle  $T_1$ has 
at most one direct successor  in $H^{'}$ other than $T_3$.

Thus, suppose $v_3\neq v_5$.
If $v_3=v_4$, then 
there cannot be a separating digon $v_2v_3v_2$, hence
the triangle 
  $v_1v_3v_5v_1$ is separating; and because of 
  $T_2\preccurlyeq v_1v_3v_5v_1 \preccurlyeq T_1$, we have $T_1=v_1v_3v_5v_1$. Thus in this case, contracting the triangle $v_0v_1v_2v_0$ will not  transform $Q_1$ into a  direct successor of $T_1$ in $H^{'}$.
  Therefore assume that $v_3\neq v_4$.

The contraction of the triangle $v_0v_1v_2v_0$ 
creates only a sequence of triangles
with pairwise containment involving the new vertex $v_1 \equiv v_2$,
apart from  $T_3$, thus preserving the invariant property.\\

\noindent
{\bf Subcase 2.3} $T_2$ 
shares the edge $v_1v_3$ with $T_3$, and $E(T_1)\cap E(T_i)=\emptyset,\ 
i=2,3$.\\

 In this subcase, every quadrilateral $Q_1=v_1v_2v_4v_5v_1$ 
containing $v_3$ in its interior also contains $T_3$. So the contraction of $v_0v_1v_2v_0$  yields two sequences of triangles with pairwise 
containment involving $v_1 \equiv v_2$: one
sequence containing $v_3$ (and also $T_3$) inside and the other sequence not containing 
$v_3$ inside, again 
preserving the property
that no triangle has three direct successors. Theorem~\ref{LE:1} now follows.
}\end{proof}

Proposition~\ref{PR:2} generalizes Proposition 2 in~\cite{Feder}.
\begin{proposition}
\label{PR:2}
 Suppose  $G$  is a $3-$connected cubic planar graph with 
 a facial $2-$factor $\mathcal{Q}$. Assume that the faces not in $\mathcal{Q}$ are either quadrilaterals or 
 hexagons, while the  faces in $\mathcal{Q}$ are
arbitrary. Suppose the reduced graph $H=G/\mathcal{Q}$  satisfies the invariant property, and  that the 
outer face of $H$ is a triangle. If $H$ has an odd number
of vertices, then $H$ has a spanning tree of faces that are triangles, and so $G$ is hamiltonian.
\end{proposition}

\begin{proof} Let $T$ be the outer face of $H$. Apply Theorem~\ref{LE:1} repeatedly to contract triangular 
bounded  faces inside of $T$ to single vertices while
preserving the invariant property. Each step reduces the number of vertices by two, so
at each step the order of the resulting graph remains odd until we are left with just the outer face $T$
and parallel edges but no vertex in $int(T)$. 
We claim that  the triangle corresponding to the
innermost face $T_0$ inside of $T$
involving all three
vertices together with the  triangles contracted in this process forms
a set of faces $\mathcal{T}$ of $H$ and defines
a spanning tree of faces $H_{\mathcal{T}}$.
Now, 
$$V(H)\setminus V(T)\subset \bigcup_{F\in \mathcal{T}-T_0} V(F) $$
guarantees that $\mathcal{T}$ covers all of $V(H)$, and 
$H_{\mathcal{T}}$ is connected by construction.

 If $H_{\mathcal{T}}$ is not a spanning tree of faces, then there 
exists a set of triangles  $\{T_1,\ldots,T_k\}$ $\subset \mathcal{T}$   
such that $|V(T_i)\cap V(T_j)|=1$ if $j=i\pm 1$, counting modulo $k$, and 
$V(T_i)\cap V(T_j)=\emptyset$ otherwise. Assume $T_{i_0}$ is 
the last contracted triangle in the contraction process of the $T_i$'s, $1\le i\le k$.
Thus after the contraction of $T_i$ for all $1\le i\neq i_0\le k$, $T_{i_0}$ 
is being transformed into a digon. This contradicts the selection of $T_{i_0}$ by 
Theorem~\ref{LE:1}.
Proposition~\ref{PR:2} now follows.
\end{proof}

We note in passing 
that by using Lemma~\ref{LEM:cyclic} below and Theorem~\ref{TH:Lf-3-conn.}(ii), Theorem~\ref{TH:LF-ham.} 
can be shown to be a special case
of Proposition~\ref{PR:2}.
Moreover, let $G$ be a $3-$connected cubic graph and as described in Theorem~\ref{TH:vertexenvelope}.  Then $H=G/\mathcal{Q}$ is a   triangulation of the plane. Thus,
by
Proposition~\ref{PR:2} and Theorems~\ref{TH:Lf-3-conn.}(ii) and~\ref{TH:vertexenvelope}, we have the following corollary.
\begin{corollary}\label{cor:2-conn.}
Let $G_0$ be a  simple $2-$connected  cubic planar
graph of order $n \equiv 2\pmod{4}$
and let $\mathcal{Q}$ be the set of faces of
$Lf(G_0)$ corresponding to faces of $G_0$. If $Lf(G_0)/\mathcal{Q}$ satisfies the invariant property,
then $Lf(G_0)$ is hamiltonian.
\end{corollary}
Note that $G_0$ has an odd number of faces if $n \equiv 2\pmod{4}$ 
where $n$ is the order of $G_0$ and thus for $G=Lf(G_0)$ we have that $H=G/\mathcal{Q}$ is of odd order.
Satisfying the invariant property is an essential condition in Corollary~\ref{cor:2-conn.}. As shown in Figure~\ref{FIG:NoQuasiSpanningTree}, for a    simple $2-$connected  cubic planar
graph  $G$,  in $H=Lf(G)/\mathcal{Q}$ the triangle $v_1v_2v_3v_1$ has seven direct successors, but as we show below, $H$ has no spanning tree of faces nor a quasi spanning tree of faces.

\begin{lemma}\label{lem:degree 4 vertex}
Let $H$ be a plane graph with the outer face $T=v_1v_2v_3v_1$ 
being triangular and
such  
that every face of  $H$ is a digon or a triangle. Suppose that  $H$ 
is $4-$connected
and that there is a vertex $v_0$ of degree $4$  inside of $T$ which   belongs to $4$ triangles at most 
one of which shares an edge with  $T$.
 If  $v_0v_4v_5v_0$ and 
$v_0v_6v_7v_0$ share no edge with each other nor with $T$ where 
$O^{+}(v_0)=\langle v_4,v_5,v_6,v_7\rangle$, then the  graph $H^{'}$ 
resulting from removing $v_0$, identifying $v_4$ with $v_5$ and identifying $v_6$ with $v_7$, satisfies the invariant property.
\end{lemma}

\begin{proof}
Since $H$ is $4-$connected and   $v_0$ is a vertex of degree $4$ inside of $T$, therefore $v_0$  belongs to $4$ triangles at most one of which shares an edge with $T$. Therefore  there exist triangles $v_0v_4v_5v_0$ and $v_0v_6v_7v_0$ which  share no edge with each other nor with $T$. Thus $H^{'}$ is well defined. By 4-connectivity of $H$, it
 has no separating digon nor a separating triangle.
 Note that
$H^{'}$ has no separating digon; otherwise, $H$ has a separating 
triangle, which is a contradiction. We show that every triangle in 
$H^{'}$ has at most two direct successors.

 We first observe that there cannot exist simultaneously two quadrilaterals 
$Q=v_0v_4v_8^{*}v_6v_0$ and $Q^{'}=v_0v_5v_9^{*}v_7v_0$ with $int(Q)\subset {{T}},\ int(Q^{'})\subset {{T}}$, each containing a vertex inside, 
$x$ and $x^{'}$, respectively, 
other than $v_4,v_5,v_6$, $v_7$, and   
$\{v_4,v_5,v_6,v_7\}\cap \{v_8^{*},v_9^{*}\}=\emptyset$.
 Otherwise, $v_8^{*}=v_9^{*}$, in which case 
there is a separating triangle containing $x$ or $x^{'}$ in  $H$, contradicting that $H$ is $4-$connected. 
Thus, without loss of generality, suppose  $Q^{'}$ does not exist. (Note that the quadrilateral $v_4v_iv_6v_{13}v_4$, $i\in \{5,7\},$  either is contained in  the quadrilateral $v_0v_4v_{13}v_6v_0$  or contains  the quadrilateral $v_0v_4v_{13}v_6v_0$,
and also  the quadrilateral $v_5v_jv_7v_{15}v_5$, $j\in \{4,6\}$, either is contained in    the quadrilateral $v_0v_5v_{15}v_7v_0$ or contains the quadrilateral $v_0v_5v_{15}v_7v_0$.)

There may, however, appear separating triangles inside of $T$ in $H^{'}$.
Such triangles derive from the following quadrilaterals. 
Let $T_{Q}\subset H^{'}$ be the triangle   deriving from the 
quadrilateral $Q\subset H$.
\begin{itemize}
\item $Q_1=v_4v_5v_8v_9v_4$ 
with $\{v_0,v_6,v_7\}\cap int(Q_1)=\emptyset$ and $\{v_0,v_6,v_7\}\cap \{v_8,v_9\}=\emptyset$ but $V(H)\cap int(Q_1)\ne\emptyset$.

\item $Q_{1^{'}}=v_4v_5v_{8^{'}}v_{9^{'}}v_4$ with $v_0 \in int(Q_{1^{'}})$  and  
$\{v_6,v_7\}\cap \{v_{8^{'}},v_{9^{'}}\}=\emptyset$.

\item $Q_2=v_6v_7v_{10}v_{11}v_6$ 
with $v_0 \notin int(Q_2)$ but $V(H)\cap int(Q_2)\ne \emptyset $ and
  such that\\
$\{v_0,v_4,v_5\}\cap \{v_{10},v_{11}\}=\emptyset$.

\item $Q_{2^{'}}=v_6v_7v_{{10}^{'}}v_{{11}^{'}}v_6$ with $v_0 \in int(Q_{2^{'}})$  and  
$\{v_4,v_5\}\cap \{v_{{10}^{'}},v_{{11}^{'}}\}=\emptyset$.

\item $Q_3=v_0v_4v_{12}v_{6}v_0$ containing  $v_7$ and at least another vertex inside.

\item $Q_{3^{'}}=v_0v_4v_{{12}^{'}}v_{6}v_0$ containing $v_5$ and at least another vertex inside. 
\end{itemize}

\noindent
Note that 

 \noindent
{\it  $H$ cannot contain 
  two quadrilaterals 
$Q_{2^{'}}$ and $Q_3$ simultaneously, and it also  cannot contain 
two quadrilaterals 
$Q_{1^{'}}$ and $Q_{3^{'}}$ simultaneously;\hfill $(*)$}\\

\noindent
otherwise, either $v_{12}=v_{{10}^{'}}$ and $H$ contains a separating 
triangle containing a vertex in $int(Q_3)$ other than $v_7$, or $v_{12}=v_{{11}^{'}}$ and $H$ contains a separating digon
$v_6v_{{11}^{'}}v_6$, which is a contradiction.

 No quadrilateral $Q\in \{Q_1,Q_{1^{'}},Q_2,Q_{2^{'}},Q_3,Q_{3^{'}}\}$
 contains a  chord inside; otherwise, there would be
a separating triangle inside of $T$, which is a contradiction.
 This implies that for all such quadrilaterals $Q^{*}$ and 
 $Q^{**}$ of the same type as $Q$, we have either $Q^{*}\subset Q^{**}$
or $Q^{**}\subset Q^{*}$. So let $Q^{'}$ be the quadrilateral of the same type as $Q$  containing all  quadrilaterals of its type, for 
each $Q\in \{Q_1,Q_{1^{'}},Q_2,Q_{2^{'}},Q_3,Q_{3^{'}}\}$.
\\

\noindent
{\it Note that $Q_2^{'}\subset Q_3^{'}\subset Q_{1^{'}}^{'}$, and
symmetrically,
$Q_1^{'}\subset Q_{3^{'}}^{'}\subset Q_{2^{'}}^{'}$.\hfill $(**)$}\\

Now we have to consider  the following cases.\\

\noindent
{\bf Case 1}
There exist the quadrilaterals $Q_3^{'}$ and $Q_{3^{'}}^{'}$ simultaneously.
\\

 In this case, by $(*)$, the graph $H$ has  no $Q_{1^{'}}^{'}$ nor 
 $Q_{2^{'}}^{'}$. Thus by $(**)$, $T$ has two direct successors 
 $T_{Q_3^{'}}$ and $T_{Q_{3^{'}}^{'}}$ in $H^{'}$.
Every triangle  of $H^{'}$ deriving from a quadrilateral of the same type 
as $Q_1$ (or symmetrically,  of the same type 
as $Q_2$)   has  at most one direct successor
deriving from  a quadrilateral of its type.
Every triangle deriving from a quadrilateral of the same type 
as $Q_3$ (or symmetrically, of the same type 
as  $Q_{3^{'}}$)   has at most one direct successor
deriving from either a quadrilateral of its type or a quadrilateral of the same  type as $Q_2$ (or of the same type 
as  $Q_1$). Thus in Case 1, $H^{'}$  satisfies the invariant property. \\

\noindent
{\bf Case 2}
There exists the quadrilateral $Q_3^{'}$ but no $Q_{3^{'}}^{'}$.
\\

In this case, by $(*)$, there is no  $Q_{2^{'}}^{'}$.
Thus by $(**)$, $T$ has at most two direct successors: they are either
 $T_{Q_1^{'}}$ and $T_{Q_{1^{'}}^{'}}$ or $T_{Q_1^{'}}$ and $T_{Q_3^{'}}$.
 Every triangle deriving from a quadrilateral of the same type 
as $Q_1$ (or symmetrically,  of the same type 
as $Q_2$)   has  at most one direct successor
deriving from  a quadrilateral of its type.
Every triangle deriving from a quadrilateral of the same type 
as  $Q_{1^{'}}$   has  at most one direct successor
deriving from  either a quadrilateral of its type or 
a quadrilateral of type $Q_3$. Every triangle deriving from a quadrilateral of the same type 
as $Q_3$  has at most one direct successor
deriving from either a quadrilateral of its type or a quadrilateral of the same  type as $Q_2$. Thus in this case, $H^{'}$  satisfies the invariant property. \\ 
If there is a quadrilateral $Q_{3^{'}}^{'}$ but no  $Q_3^{'}$, we argue analogously.\\
 
\noindent
{\bf Case 3}
$T$ contains neither $Q_3^{'}$  nor $Q_{3^{'}}^{'}$.
\\

In this case, by  $(**)$ we have $Q_1^{'}\subset Q_{2^{'}}^{'}$ and
$Q_2^{'}\subset Q_{1^{'}}^{'}$. So $T$ has at most two direct successors
in $H^{'}$.
 Every triangle deriving from a quadrilateral of the same type 
as $Q_1$ (or symmetrically,  of the same type 
as $Q_2$)   has  at most one direct successor
deriving from  a quadrilateral of its type.
Every triangle deriving from a quadrilateral of the same type 
as $Q_{1^{'}}$ (or of the same type 
as $Q_{2^{'}}$)   has at most one direct successor
deriving from either a quadrilateral of its type or a quadrilateral of the same  
type as $Q_2$ (or symmetrically,  of the same type 
as $Q_1$). Therefore also in Case $3$,  $H^{'}$  satisfies the invariant property. Lemma~\ref{lem:degree 4 vertex} now follows.
\end{proof}

In the case of $H$ having an even number of vertices,  we 
are now able to find a 
quasi spanning tree of faces in $H$ provided $H$ has a degree $4$ 
vertex.

\begin{proposition}\label{PR:even-order}
Consider $G$ and $\mathcal{Q}$ as in Proposition~$\ref{PR:2}$.
  Suppose that the reduced graph $H=G/\mathcal{Q}$ is $4-$connected
  and that  the outer face $T$ of $H$ is triangular.
 If $H$ has an even number
of vertices, and such that there is a vertex of degree $4$  in $int(T)$,
 then $H$ has a quasi spanning tree of faces which are triangles, and so $G$ is hamiltonian.
\end{proposition}

\begin{proof}
Note that the graphs $H$ under consideration satisfy the invariant property in a more restricted way (since by $\kappa(H)=4$, 
the graph $H$ has no separating digon nor a separating triangle).
Let $T=v_1v_2v_3v_1$ be the outer face of $H$ and let $v_0$ be a vertex 
of degree $4$ in $int(T)$.  By the hypothesis 
(no separating digon nor a separating triangle and $|V(H)|$ is even), $v_0$ cannot be incident to multiple edges
unless $K_4$ spans $H$ and without loss of generality, $v_0v_3$ 
is a multiple edge
in which case $H$ is not $4-$connected. Nonetheless in
 this exceptional case the two triangular faces 
$v_0v_1v_3v_0$ and $v_0v_2v_3v_0$ define a  quasi spanning tree of faces of $H$ with the 
quasi vertex $v_0$. Therefore, in what follows we may assume that
the vertex $v_0$ belongs to $4$ triangles at most 
one of which shares an edge with $T$
and proceed as in the proof of Lemma~\ref{lem:degree 4 vertex}. Set $N(v_0)=\{v_4,v_5,v_6,v_7\}$ 
and $O^{+}(v_0)=\langle v_4,v_5,v_6,v_7\rangle$.

 Select two triangles involving $v_0$
that do not share an edge with each other nor with $T$, say $v_0v_4v_5v_0$ and 
$v_0v_6v_7v_0$.

Let $H^{'}$ be the graph obtained from $H$ by removing $v_0$, 
identifying $v_4$ with $v_5$, and identifying $v_6$ with $v_7$.

Clearly, $H^{'}$ has an odd number of vertices and by Lemma~\ref{lem:degree 4 vertex}, $H^{'}$ has the invariant property. Thus by
 Proposition~\ref{PR:2}, $H^{'}$ has a spanning tree of faces
 $H_{\mathcal{T}^{'}}^{'}$  where all elements of $\mathcal{T}^{'}$   are triangles. It is easy to see that 
 the union of $v_0v_4v_5v_0$ and 
$v_0v_6v_7v_0$ with the corresponding faces of $\mathcal{T}^{'}$ in $H$ form 
a set of faces $\mathcal{T}$ and 
a quasi spanning tree of faces $H_{\mathcal{T}}$ in $H$ and $v_0$ is a quasi 
vertex  of $H_{\mathcal{T}}$.
\end{proof}

\medskip \noindent
Since  every simple $4-$connected eulerian triangulation of the plane has  
at least six vertices of degree $4$,
 the following is an immediate corollary of Propositions~\ref{PR:2} and~\ref{PR:even-order}.

\begin{corollary}\label{cor:4con-eulerian-triangulation}
Every simple $4-$connected eulerian triangulation of the plane has a 
quasi  spanning tree of faces.
\end{corollary}

\begin{example}
We claim that the  $3-$connected triangulation of the plane of Figure~$\ref{FIG:NoQuasiSpanningTree}$ below has no quasi spanning tree of 
faces.
\begin{figure}[ht]

\setlength{\unitlength}{0.15cm}
\vspace{1cm}
\begin{center}

\begin{picture}(40,30)
\put(2,25){$H$}
\put(0,0){\circle*{1.2}}
\put(-3.5,0){\footnotesize$v_1$}
\put(40,0){\circle*{1.2}}
\put(41,0){\footnotesize$v_3$}
\put(20,34.6){\circle*{1.2}}
\put(19,36){\footnotesize$v_2$}

\put(20,7.05){\circle*{1.2}}
\put(20.5,5.2){\footnotesize$v_{4}$}
\put(24,13.6){\circle*{1.2}}
\put(21,14.1){\footnotesize$v_6$}
\put(16,13.6){\circle*{1.2}}
\put(17,14.1){\footnotesize$v_5$}

\put(20,3.52){\circle*{1.2}}
\put(19,1.5){\footnotesize$v_{12}$}

\put(27,15.5){\circle*{1.2}}
\put(27.5,15.4){\footnotesize$v_{10}$}

\put(13,15.5){\circle*{1.2}}
\put(10,15.4){\footnotesize$v_8$}

\put(20,11.4){\circle*{1.2}}
\put(17.65,10.2){\footnotesize$v_{0}$}

\put(12,6.8){\circle*{1.2}}
\put(13,7.3){\footnotesize$v_7$}

\put(28,6.8){\circle*{1.2}}
\put(24,7.3){\footnotesize$v_{11}$}

\put(20,21){\circle*{1.2}}
\put(19,18){\footnotesize$v_9$}

\qbezier(0,0)(0,0)(40,0)
\qbezier(0,0)(0,0)(20,34.6)
\qbezier(0,0)(0,0)(20,7.05)
\qbezier(0,0)(0,0)(16,13.6)
\qbezier(0,0)(0,0)(20,3.52)
    \qbezier(20,3.52)(20,3.52)(40,0)
    \qbezier(20,3.52)(20,3.52)(20,7.05)
\qbezier(0,0)(0,0)(12,6.8)
    \qbezier(12,6.8)(12,6.8)(20,7.05)
    \qbezier(12,6.8)(12,6.8)(16,13.6)
\qbezier(0,0)(0,0)(13,15.5)
    \qbezier(13,15.5)(13,15.5)(16,13.6)
    \qbezier(13,15.5)(13,15.5)(20,34.6)

\qbezier(40,0)(40,0)(20,34.6)
\qbezier(40,0)(40,0)(20,7.05)
\qbezier(40,0)(40,0)(24,13.6)
\qbezier(40,0)(40,0)(28,6.8)
    \qbezier(28,6.8)(28,6.8)(20,7.05)
    \qbezier(28,6.8)(28,6.8)(24,13.6)
\qbezier(40,0)(40,0)(27,15.5)
    \qbezier(27,15.5)(27,15.5)(24,13.6)
    \qbezier(27,15.5)(27,15.5)(20,34.6)

\qbezier(20,34.6)(20,34.6)(16,13.6)
\qbezier(20,34.6)(20,34.6)(24,13.6)
\qbezier(20,34.6)(20,34.6)(20,21)
    \qbezier(20,21)(20,21)(24,13.6)
    \qbezier(20,21)(20,21)(16,13.6)

\qbezier(20,7.05)(20,7.05)(16,13.6)
\qbezier(20,7.05)(20,7.05)(24,13.6)
\qbezier(16,13.6)(16,13.6)(24,13.6)

\qbezier(20,11.4)(20,11.4)(20,7.05)
\qbezier(20,11.4)(20,11.4)(24,13.6)
\qbezier(20,11.4)(20,11.4)(16,13.6)

\end{picture}
\end{center}
\caption{\small\it 
 A $3-$connected triangulation $H$ of the plane without quasi spanning tree
 of faces.}
\label{FIG:NoQuasiSpanningTree}
\end{figure}
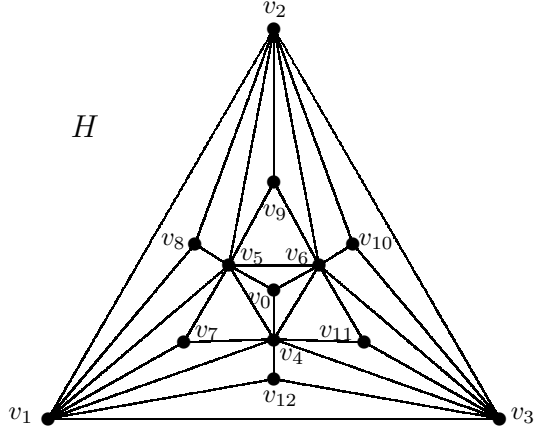

Proceeding by contradiction, we first assume that there is a set $\mathcal{T}$  of faces and $H_{\mathcal{T}}$ is a spanning tree of faces in 
$H$. For every degree three vertex $v_0$ and $v_i$, $7\le i\le 12$, there exists 
precisely one triangle in $\mathcal{T}$ containing $v_0$ or $v_i$. Without loss of 
generality and because of symmetry, $\{v_{0}v_4v_5v_0,v_1v_7v_5v_1\}\subset \mathcal{T}$. Then 
$v_1v_5v_8v_1\notin \mathcal{T}$, so $v_1v_2v_8v_1$ or 
$v_2v_5v_8v_2\in \mathcal{T}$. Since $H_{\mathcal{T}}$ has no quasi vertex, 
and because no two faces in $\mathcal{T}$ share  an edge, therefore
$v_2v_5v_9v_2\notin \mathcal{T}$ and  so $v_2v_6v_9v_2$ or $v_5v_6v_9v_5$ belongs to $\mathcal{T}$, thus as a consequence  $v_4v_6v_{11}v_4\notin \mathcal{T}$.
Again since $H_{\mathcal{T}}$ has no quasi vertex, 
$v_1v_4v_{12}v_1\notin \mathcal{T}$ and so $v_1v_3v_{12}v_1$ or 
$v_3v_4v_{12}v_3\in \mathcal{T}$. 
Therefore, $\{v_3v_4v_{11}v_3,v_3v_6v_{11}v_3\}\cap \mathcal{T}=\emptyset\ ($otherwise, there is a cycle of faces in $H_{\mathcal{T}})$.
Thus, there is no face  in $\mathcal{T}$ containing $v_{11}$, which is 
a contradiction. By a similar argument one can show that $H$ has no 
quasi spanning tree of faces,
observing that quasi vertices must have even degree and thus without loss of generality, $v_5$ would be a quasi vertex and
 $\{v_0v_4v_5v_0,v_1v_5v_7v_1,v_2v_5v_8v_2,v_5v_6v_9v_5\}\subset \mathcal{T}$;
 and as a consequence, $v_4$ and $v_6$ must be proper vertices; otherwise, $v_1v_4v_7v_1 \in \mathcal{T}$ or $v_0v_4v_6v_0 \in \mathcal{T}$, respectively, which is a contradiction. 
\end{example}

Corollary~\ref{cor:4con-eulerian-triangulation} implies a result on hamiltonicity in 
planar cubic bipartite graphs.

\begin{theorem}\label{TH:4conn}
Let $G$ be a bipartite cubic  planar graph with the following properties.

\item[$(i)$]
In the natural $3-$face coloring  of $G$ with colors $1,2,3$, two of the 
color classes $($without loss of generality, color classes $\mathcal{C}_1$ 
and $\mathcal{C}_2)$
contain hexagons only.

\item[$(ii)$] The contraction  of the faces in color class $\mathcal{C}_3$ is $4-$connected.

Then $G$ is hamiltonian. 
\end{theorem}

\begin{lemma}\label{LEM:cyclic}
Let $G$ be a simple  cubic  planar graph and let $\mathcal{Q}$ be 
the set of faces in $Lf(G)$ corresponding to the faces of $G$. Then,
$$\kappa^{'}_c(G)=\kappa(Lf(G)/\mathcal{Q}).$$
\end{lemma}

\begin{proof}
Let $H=Lf(G)/\mathcal{Q}$.
Note that by Definition~\ref{DEF:Leapfrog}~(ii), the reduced graph $H$ is
a triangulation of the plane and every edge of $H$ corresponds to 
a unique edge of $G$, and vice versa; and every vertex of $H$ corresponds 
to  a unique face of $G$, and vice versa.
Note that $G$ and $H$ can be drawn in the plane in such a way that 
$f\in V(H)$ lies in $int(F)$ where $f$ corresponds to the face $F\in \mathcal{F}(G)$,
and such that $ff^{'}\in E(H)$ crosses the corresponding edge 
$e\in E(bd(F))\cap E(bd(F^{'}))\subset E(G)$ precisely once.

Suppose that $X\subset V(H)$, $|X|=k$, is a minimum vertex cut in $H$. Since $H$ 
is a 
triangulation of the plane, the induced subgraph 
$\langle X\rangle_H$ is a  cycle
 $C=f_1f_2\ldots f_{k}f_1$ 
such that $int(C)\cap V(H)\neq \emptyset \neq ext(C)\cap V(H)$.
Denote some vertices $f_{k+1}\in int(C)\cap V(H)$ and 
$f_{k+2}\in ext(C)\cap V(H)$.

Denote by $v_iv_j\in E(G)$  the edge corresponding to the  edge
$f_if_j\in E(C)$, $1\le i,j\le k$.
Then, $Y=\{v_iv_{i+1}\ |\ 1\le i\le k-1\}\cup \{v_1v_{k}\}$ 
separates in $G$ the face boundaries whose corresponding vertices 
in $V(H)$ lie in $int(C)$ from the face boundaries  whose 
corresponding vertices in $V(H)$ lie in $ext(C)$.
 Thus, $Y$ is a 
cyclic edge cut of $G$ and therefore, $\kappa(H) \ge \kappa^{'}_c(G)$.
By an analogous argument we obtain  $\kappa(H) \le \kappa^{'}_c(G)$;
hence, $\kappa(H)= \kappa^{'}_c(G)$.
\end{proof}

In the graph $G$ as stated in Lemma~\ref{LEM:cyclic},  color  the faces 
in $\mathcal{Q}$ with  color $3$. Then by Theorem~\ref{TH:4conn} and Lemma~\ref{LEM:cyclic}, we obtain the following corollary.

\begin{corollary}\label{cor:cycl.4-edge-conn}
Let $G_0$ be a  cyclically $4-$edge-connected bipartite cubic planar
graph. Then the 
leapfrog extension of $G_0$ is hamiltonian.
\end{corollary}
Again let $G$ be a  graph as stated in Theorem~\ref{TH:vertexenvelope}.  Then $H=G/\mathcal{Q}$ is a   triangulation of the plane. Thus by
applying Lemma~\ref{LEM:cyclic} and Proposition~\ref{PR:1}, we may also conclude that
Corollary~\ref{cor:4con-eulerian-triangulation} implies Corollary~\ref{cor:cycl.4-edge-conn}. 
\\

We note in passing that these results together with Theorem~\ref{TH:LF-ham.} are the best partial solutions of 
Barnette's Conjecture, so far.

\begin{theorem}
\label{TH:1}
Let  $G$ be a planar $3-$connected cubic graph with a facial $2-$factor $\mathcal{Q}$. Suppose all $\mathcal{Q}^{c}-$faces of $G$ are either quadrilaterals or 
 hexagons, 
while the $\mathcal{Q}-$faces are arbitrary.
Assume  the outer face of the reduced
graph $H$ obtained from $G$ by the contraction of the $\mathcal{Q}-$faces is a triangle $T$,
and assume that $T$ and every triangle in  $H$ has an
even number of vertices in its interior. 
 If every direct successor in $H$ contains no separating
  digon $($if a direct successor exists$)$, then $H$ has a  spanning tree of faces that are triangles, yielding a
hamiltonian cycle for $G$.
\end{theorem}

\begin{proof}{
Consider $H$; it is of odd order $n$.
We proceed by induction on ${n}$.
For $n=3$ since $G$ is $3-$connected, $H$ is a triangle with some parallel edges. Note that, $H$ has no separating triangle,
thus by Proposition~\ref{PR:2}, $H$ has a spanning tree of faces.

Assume that the theorem is true for every graph of odd order less than $n$ satisfying the hypothesis.

If $H$ has no separating triangle, then by Proposition~\ref{PR:2}, $H$ has a spanning tree of faces. Thus, there is  a separating triangle $T_1$ in $H$ such that no triangle inside of $T_1$ is  separating.
Therefore,  $|int(T_1)\cap V(H)|\ge 2$   but contains no 
separating triangle nor a separating digon.
Thus, $T_1$ satisfies the invariant property.
Therefore by Theorem~\ref{LE:1},  there exists a triangular face $T^{'}$  such that
$int(T^{'})\subset int(T_1)$ and $|V(T_1)\cap V(T^{'})|\le 1$,  and after contracting $T^{'}$ to a single vertex, $T_1$ will satisfy 
the invariant property. 
Now, let $H^{'}$ be the graph obtained from $H$ by contracting ${T^{'}}$.

It is easy to see that $H^{'}$ satisfies all hypotheses of Theorem~\ref{TH:1} and its order is $n-2$.
(Note that every separating digon in $H^{'}$ would derive from a separating triangle inside $T_1$ in $H$, contrary to the choice of $T_1$). Thus by induction,
$H^{'}$ has a  spanning tree of faces that are triangles with face  set $\mathcal{T}^{'}$. 
Let $\mathcal{T}$ be the union of the set of the corresponding faces of $\mathcal{T}^{'}$ in $H$ and the set  ${\{T^{'}\}}$.
It can be easily seen that $H_\mathcal{T}$ is a  spanning tree of faces in $H$. This completes the proof of Theorem~\ref{TH:1}.
}\end{proof}

We note finally  that in~\cite{FleischnerEnvelope},
hamiltonicity in the leapfrog extension of a plane cubic graph was studied from a different point of view.


\begin{thebibliography}{99}
%
\bibitem{Barenette}{ D. W. Barnette}, {\it Conjecture $5$}.  In
William T. Tutte,  ed.,
{\it Recent Progress  in  Combinatorics.  Proceedings  of  the  3rd  Waterloo  Conference  on
Combinatorics, vol.} 3, p. 343. 1969.

\bibitem{Biggs}
{\it N.L. Biggs, E.K. Lloyd and R.J. Wilson}, Graph theory: $1736-1936$, {\em Clarendon Press},  Oxford, 1976.

\bibitem{Bondy}
{\it J.A. Bondy and U.S.R. Murty}, Graph theory, {\em Graduate Texts in Mathematics,  {\bf 244}, Springer}, New York, 2008.

\bibitem{Feder}{T. Feder and C. Subi},
{\it  On Barnette's Conjecture}, Electronic Colloquium on Computational Complexity TR06-015 (2006).

\bibitem{Fleischner}
{\it H. Fleischner}, Eulerian graphs and related topics, {\em Part 1, Volume 1, Ann. of Discrete Math.  {\bf 45}}, 1990.

\bibitem{FleischnerEnvelope}
{\it H. Fleischner, A.M. Hobbs, M.  Tapfuma Muzheve}, Hamiltonicity in vertex envelopes of plane cubic graphs, {\em  Discrete Math.}, 
{\bf 309(14)} (2009), 4793--4809.

\bibitem{Fowler}
{\it P.W. Fowler}, How unusual is $C_{60}$? Magic numbers for carbon clusters, {\em Chem. Phys. Lett.}, 
{\bf 131(6)} (1986), 444--450.


\bibitem{Garey}
{\it M.R. Garey, D.S. Johnson, and R.E. Tarjan}, The planar Hamiltonian 
circuit problem is NP-complete, {\em SIAM J. Comput.} {\bf 5} (1976), 704--714.

\bibitem{Goodey}
{\it P.R. Goodey}, Hamiltonian circuits in polytopes with even sides, {\em Israel Journal of
Mathematics} {\bf 22} (1975), 52--56.

\bibitem{Holton}
{\it D.A. Holton, B. Manvel, and B.D. McKay}, Hamiltonian cycles in cubic 3-connected
bipartite planar graphs, {\em Journal of Combinatorial Theory B} {\bf 38} (1985), 279--297.

\bibitem{Horton}
{\it J.D. Horton}, On two-factors of bipartite regular graphs, {\em Discrete Math.} {\bf 41(1)} (1982), 35--41.



\bibitem{Kardos}
{\it F. Kardo$\check{s}$}, A computer-assisted proof of Barnette-Goodey
conjecture: Not only fullerene graphs are Hamiltonian, 
{\href{https://arxiv.org/pdf/1409.2440v2.pdf}{\bf arXiv:1409.2440v2}}.


\bibitem{Kirkman}
{\it T.P. Kirkman}, On the presentation of polyhedra, 
{\em Philos. Trans. Roy. Soc.  {\bf 146}}, (1856), 413--418.

\bibitem{Payan}
{\it C. Payan, C. and M. Sakarovitch}, Ensembles cycliquement stables et graphes cubiques,
{\em Cahiers Centre D'\'{e}tudes Recherche Op\'{e}r.  {\bf 17}},  (1975),  319--343.

\bibitem{Petersen}
{\it J. Petersen}, Sur le th$\rm\acute{e}$or$\rm\grave{e}$me de Tait, 
{\em Interm$\acute{e}$d. Math. {\bf 5}}, (1898), 225--227.

\bibitem{Tait}
{\it P.G. Tait}, Listing's Topologie, {\em Philosophical Magazine} (5th ser.), {\bf 17} (1884), 30--46. Reprinted in Scientific Papers, Vol. II, pp. 85--98.


\bibitem{Takanori}
{\it A. Takanori, N. Takao, and S. Nobuji}, NP-completeness of the
Hamiltonian cycle problem for bipartite graphs, {\em Journal of Information Processing
Abstract} Vol. 03  (1980) 73--76.

\bibitem{Tutte1946}
{\it W.T. Tutte}, On Hamiltonian circuits, {\em J. London Math. Soc.}, {\bf 21(2)} (1946), 98--101.

\bibitem{Tutte1971}
{\it W.T. Tutte}, On the $2-$factors of bicubic graphs, {\em Discrete Math.}, {\bf 1(2)} (1971), 203--208.


\bibitem{Yoshida}
{\it M. Yoshida, M. Fujita, P.W. Fowler, E.C. Kirby}, Non-bonding orbitals in graphite, carbon tubules, toroids and
fullerenes, {\em J. Chem. Soc., Faraday Trans.}, {\bf 93} (1997), 1037--1043.

\end{thebibliography}
\end{document}